\documentclass[12pt,reqno]{article}
\usepackage[english]{babel}
\usepackage{amsmath,amsthm}
\usepackage{amsfonts}
\usepackage{graphicx}
\usepackage{enumerate}
\usepackage[usenames,dvipsnames]{color}
\usepackage{subfigure}
\usepackage[right,pagewise,displaymath, mathlines]{lineno}
\usepackage{epstopdf}
\usepackage{color}
\usepackage{multirow}

\numberwithin{equation}{section}
\numberwithin{figure}{section}
\numberwithin{table}{section}

\newtheorem{theorem}{Theorem}[section]
\newtheorem{corollary}{Corollary}[section]
\newtheorem{lemma}{Lemma}[section]

\newtheorem{definition}{Definition}[section]
\newtheorem{example}{Example}[section]

\numberwithin{equation}{section}

\begin{document}


\begin{center}

{\bf \LARGE Searching for, and quantifying, non-convexity of functions}

\vspace*{8mm}

{\bf \large 
Youri Davydov$^{1,}\footnote{youri.davydov@univ-lille1.fr}$,
Elina Moldavskaya$^{2,}\footnote{elinamoldavskaya@gmail.com}$,
and
Ri\v cardas Zitikis$^{3,}\footnote{rzitikis@uwo.ca}$}

\bigskip

$^{1}$\textit{Chebyshev Laboratory, St.\,Petersburg State University, \break  Vasilyevsky Island, St.\,Petersburg 199178, Russia}

\medskip

$^{2}$\textit{Department of Mathematics \& Technion International School, \break Technion -- Israel Institute of Technology, Haifa 32000, Israel}

\medskip

$^{3}$\textit{School of Mathematical and Statistical Sciences,
Western University, \break London, Ontario N6A 5B7, Canada}

\end{center}

\bigskip

\begin{quote}
\noindent
\textbf{Abstract.} Convexity plays a prominent role in a number of problems, but practical considerations frequently give rise to non-convex functions. We suggest a method for determining convex regions, and also for assessing the lack of convexity in the other regions.  The method relies on a specially constructed decomposition of symmetric matrices, such as the Hessian. We illustrate theoretical results using several examples, one of which analyses a problem arising in risk measurement and management in insurance and finance.

\bigskip


\end{quote}

\newpage

\section{Introduction}
\label{introduction}

Solutions to a great variety of optimization problems rely on convexity (e.g., Boyd and Vandenberghe,~2004), but the underlying objects (e.g., functions, surfaces, etc.) may not always be such. Several techniques can be used to overcome the challenge. One of them is based on turning non-convex objects into convex ones, and a classical example would be turning data points (e.g., incomes) into the Lorenz curve, which is convex, and has been extensively used by econometricians to measure income inequality since the pioneering works of Lorenz (1905) and Gini (1912, 1914). In turn, these ideas have given rise to lift zonoids, convex hulls, and other convex objects (e.g., Mosler, 2002). In the theory of stochastic processes, convexifications of random walks have lead to multi-dimensional convex bodies (e.g., Davydov and Vershik, 1998), as well as to convex stochastic processes (e.g., Davydov and Zitikis, 2004). While dealing with such problems, a natural question arises: how far are the obtained convexifications from the original objects?

Another technique is based on working with non-convex objects directly, perhaps initially  modifying, extending, and generalizing some of the techniques developed for tackling convex cases (e.g., Mishra, 2011). In such non-convex scenarios, the underlying functions are still convex, or concave, over certain regions of their domains of definition. A natural question arises: how can we determine, and fairly quickly due to practical considerations, those regions that are convex or concave? The present paper offers answers to questions like these by providing a rigorous methodology for determining and quantifying convexity or, in a dual way, the lack of it. The method is computationally friendly and leads to numerical assessments even when closed-form solutions are difficult to derive.

We have organized the rest of the paper as follows. In Section \ref{results=1}, we present indices of lack-of-convexity, as well as their dual versions called indices of convexity. In the same section, we provide a simple example that illustrates the concept. Section \ref{results=2} contains fundamentals that concernt the lack of positive semidefiniteness in symmetric matrices. Section \ref{illustration} is devoted to a detailed analysis of a problem that has arisen in risk measurement and management. Section \ref{conclude} concludes the paper with a brief summary of main contributions.

\section{Indices of convexity}
\label{results=1}

We start with a simple yet illuminating introduction to the main idea. Let $h:(a,b)\to \mathbb{R}$ be a real-valued function on a bounded interval $(a,b) \subset \mathbb{R}$. Assume that $h$ is  differentiable, and thus it is convex on the interval $(a,b)$ if and only if its first derivative $h'$ is non-decreasing on the interval. This reduces our task to the assessment of how much the derivative is increasing, and for this, we employ the idea of Davydov and Zitikis (2017). Namely, assuming that $h'$ is differentiable, that is, the original function $h$ is an element of the space $C^2(a,b)$, the index of increase $\mathrm{INC}(h')$ of $h'$ or, in other words, the index of convexity $\mathrm{CONV}(h)$ of $h$ over the interval $(a,b)$ is
\begin{equation}\label{inc-0}
\mathrm{CONV}(h)=\mathrm{INC}(h')
:={\int_a^b (h''(x))_{+}\mathrm{d}x \over \int_a^b |h''(x)| \mathrm{d}x},
\end{equation}
where $(h''(x))_{+}=\max \{h''(x), 0\} $. Obviously, $\mathrm{CONV}(h)\in [0,1]$ for every function $h$, but if $h$ is convex on the interval $(a,b)$, then $\mathrm{CONV}(h)=1$, and if $h$ is concave, then $\mathrm{CONV}(h)=0$. This index has played a pivotal role in several applications, including financial and insurance risk management (Davydov and Zitikis, 2017), educational measurement (Chen and Zitikis, 2017), control systems assessment (Gribkova and Zitikis,~2018); we refer to Chen at al. (2018) for details and additional references on the topic. The application (details are in Section \ref{illustration}) that has inspired our present research concerns functions over multi-dimensional domains that require much more sophisticated considerations, which we describe next.

Let $h:G\to \mathbb{R}$ be a real-valued function from a bounded, open, and convex $d$-dimensional domain $G\subset \mathbb{R}^d$, for some $d\in \mathbb{N}$. Assume that the function is twice continuously differentiable, that is, $h\in C^2(G)$. Consequently, its  Hessian $H_h(\mathbf{x})$ exists at each point $\mathbf{x}=(x_1,\dots , x_d)\in G$ and is a symmetric $d\times d$ matrix. The function $h$ is convex on $G$ if and only if the Hessian $H_h(\mathbf{x})$ is positive semidefinite. Hence, the problem posited in the introduction can be viewed as an assessment problem of how much, if at all, the Hessian $H_h(\mathbf{x})$ deviates from being positive semidefinite. In view of this, we can define the index of lack of convexity of the function $h$ as a distance of the Hessian $H_h(\mathbf{x})$ from the set of all positive semidefinite, symmetric, $d\times d$ matrices. Obviously, the index is equal to $0$ whenever the function $h$ has no lack-of-convexity, that is, when the function is convex.

To proceed, we need additional notation. Let $\lambda_1(\mathbf{x}),\dots , \lambda_d(\mathbf{x})$  denote the eigenvalues of the Hessian $H_h(\mathbf{x})$. They are real because the Hessian is symmetric. We define their positive and negative parts by  $\lambda_i^{+}(\mathbf{x})=\max \{\lambda_i(\mathbf{x}), 0\} $ and $\lambda_i^{-}(\mathbf{x})=\max \{-\lambda_i(\mathbf{x}), 0\} $, respectively; for all $i=1,\dots , d$. Obviously, $\lambda_i(\mathbf{x})=\lambda_i^{+}(\mathbf{x})-\lambda_i^{-}(\mathbf{x})$ and
$|\lambda_i(\mathbf{x})|=\lambda_i^{+}(\mathbf{x})+\lambda_i^{-}(\mathbf{x})$.

\begin{definition}\label{def-0a}
The index of lack of convexity (LOC) of function $h$ at point $\mathbf{x}\in G$ is
\begin{equation}\label{loc-1}
\mathrm{LOC}(h,\mathbf{x})=\sum_{i=1}^d \lambda_i^{-}(\mathbf{x}) .
\end{equation}
\end{definition}

If the Hessian is positive semidefinite, then all its eigenvalues $\lambda_1(\mathbf{x}),\dots , \lambda_d(\mathbf{x})$  are non-negative, and thus $\mathrm{LOC}(h,\mathbf{x})=0$, which means ``zero lack of convexity.''  In other words, the function $h$ is convex at the point $x$. The value of $\mathrm{LOC}(h,\mathbf{x})$ never exceeds the nuclear (also known as the trace) norm $\Vert H_h(\mathbf{x}) \Vert_{*}=\sum_{i=1}^d |\lambda_i(\mathbf{x})|$ of the Hessian $H_h(\mathbf{x})$. This gives rise to our next definition.

\begin{definition}\label{def-0b}
The normalized LOC index of function $h$ at point $\mathbf{x}\in G$ is
\begin{equation}\label{nloc-1}
\mathrm{NLOC}(h,\mathbf{x})={\sum_{i=1}^d \lambda_i^{-}(\mathbf{x}) \over \sum_{i=1}^d |\lambda_i(\mathbf{x})|}.
\end{equation}
\end{definition}

It follows from the definition that $\mathrm{NLOC}(h,\mathbf{x})\in [0,1]$, and thus the index is normalized. Furthermore, since $\mathrm{NLOC}(h,\mathbf{x})=0$ means no lack of convexity (i.e., convexity) and $\mathrm{NLOC}(h,\mathbf{x})=1$ means total lack of convexity, we can use the quantity
$1-\mathrm{NLOC}(h,\mathbf{x})$ to measure convexity: it takes value $0$ in the case of total lack of convexity and value $1$ in the case of pure convexity. We have arrived at our third definition.

\begin{definition}\label{def-0c}
The index of convexity of function $h$ at point $\mathbf{x}\in G$ is
\begin{equation}\label{c-1}
\mathrm{CONV}(h,\mathbf{x})={\sum_{i=1}^d \lambda_i^{+}(\mathbf{x}) \over \sum_{i=1}^d |\lambda_i(\mathbf{x})|}.
\end{equation}
\end{definition}

Starting from the above introduced pointwise indices of convexity, or lack of it, we can create a variety of global indices of convexity of the function $h$ over the domain $G$. For example, the classical $L_1$-norm leads to the following global index of convexity
\begin{equation}\label{c-1b}
\mathrm{CONV}_1(h,G)
={\int_G \sum_{i=1}^d \lambda_i^{+}(\mathbf{x})\mathrm{d}\mathbf{x} \over \int_G \sum_{i=1}^d |\lambda_i(\mathbf{x})|\mathrm{d}\mathbf{x}}.
\end{equation}
The index is always in the interval $[0,1]$, takes value $0$ when the function $h$ is not convex at any point $\mathbf{x}\in G$, and takes value $1$ when it is convex at every $\mathbf{x}\in G$.

Certainly, applications may suggest using other definitions of global indices, perhaps using more complex functional norms such as that of the Lebesgue space $L_p(G,w,\mathrm{d}\mu)$ with various choices of the weight function $w$ and positive measure $\mu $. Nevertheless, for the sake of transparency, throughout the paper we use the $L_1$-norm and thus drop the subindex ``$1$'' from $\mathrm{CONV}_1(h,G)$ to simplify the notation and increase readability. Next is an illustrative example of the index $\mathrm{CONV}(h,G)$ based on a very basic yet instructive example of $h$.

\begin{example}\label{example-1}\rm
Consider the function of two separable arguments
\[
h_{\cos}(x,y)=-\cos x -\cos y
\]
on the square
\[
S_{0,0}(a)=(-a,a)\times (-a,a),
\]
which is centered at $(0,0)$ and parameterized by $a\in (0,\infty )$ that we shall vary when exploring the convexity of the function $h_{\cos}$ over $S_{0,0}(a)$. The function is depicted in Figure~\ref{graphs-0}.
\begin{figure}[h!]
\centering
\subfigure[Function $h_{\cos}(x,y)$ on the square $S_{0,0}(4)$.]{
\includegraphics[width=0.5\textwidth]{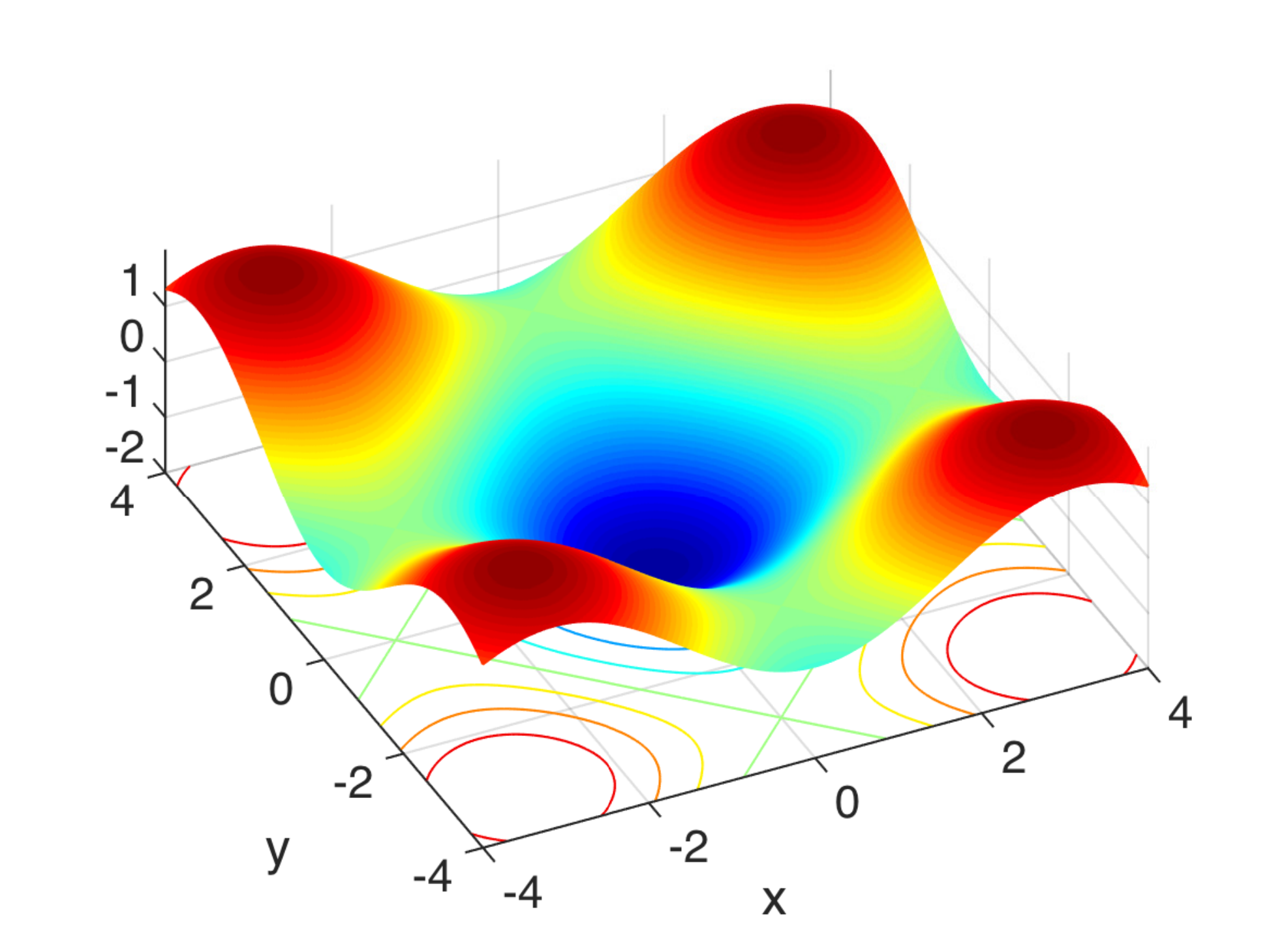}}\qquad
\subfigure[The Hessian is positive semidefinite only in the centre square.]{
\includegraphics[width=0.4\textwidth]{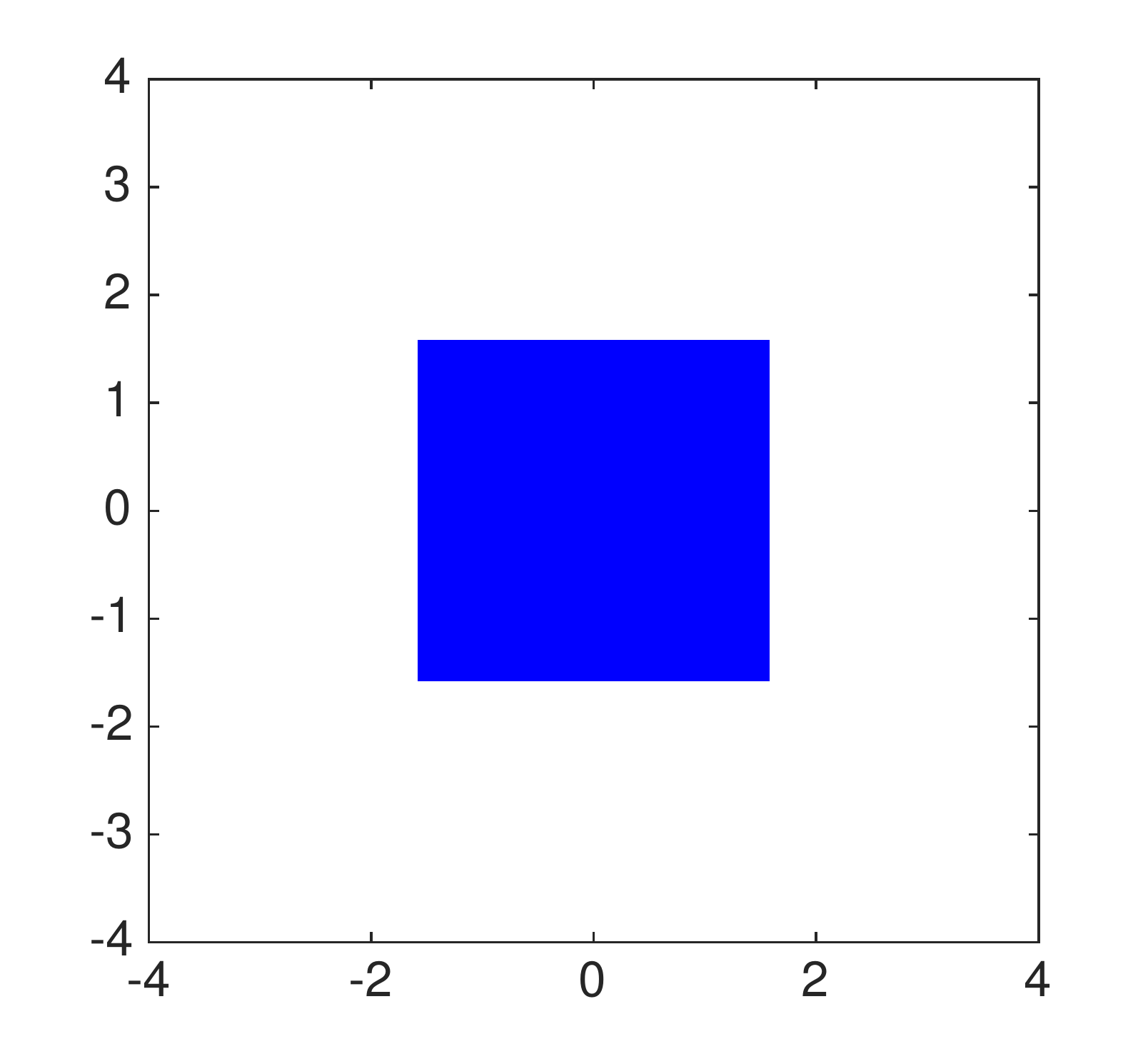}}\\
\subfigure[Pointwise convexity index for $(x,y)\in S_{0,0}(4)$.]{
\includegraphics[width=0.5\textwidth]{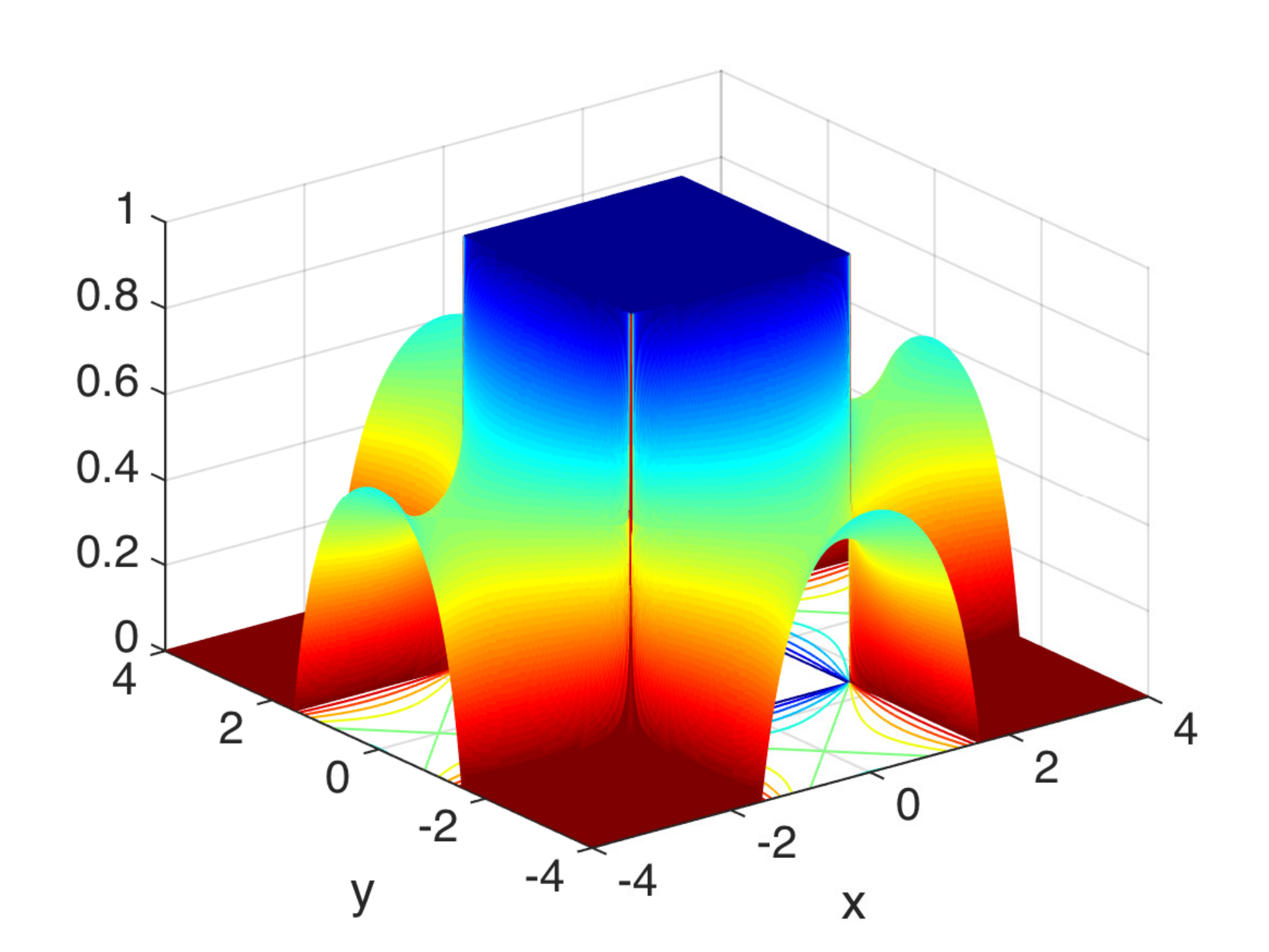}}\qquad
\subfigure[Global index $\mathrm{CONV}(a)$ for $0\le a \le 20$.]{
\includegraphics[width=0.43\textwidth]{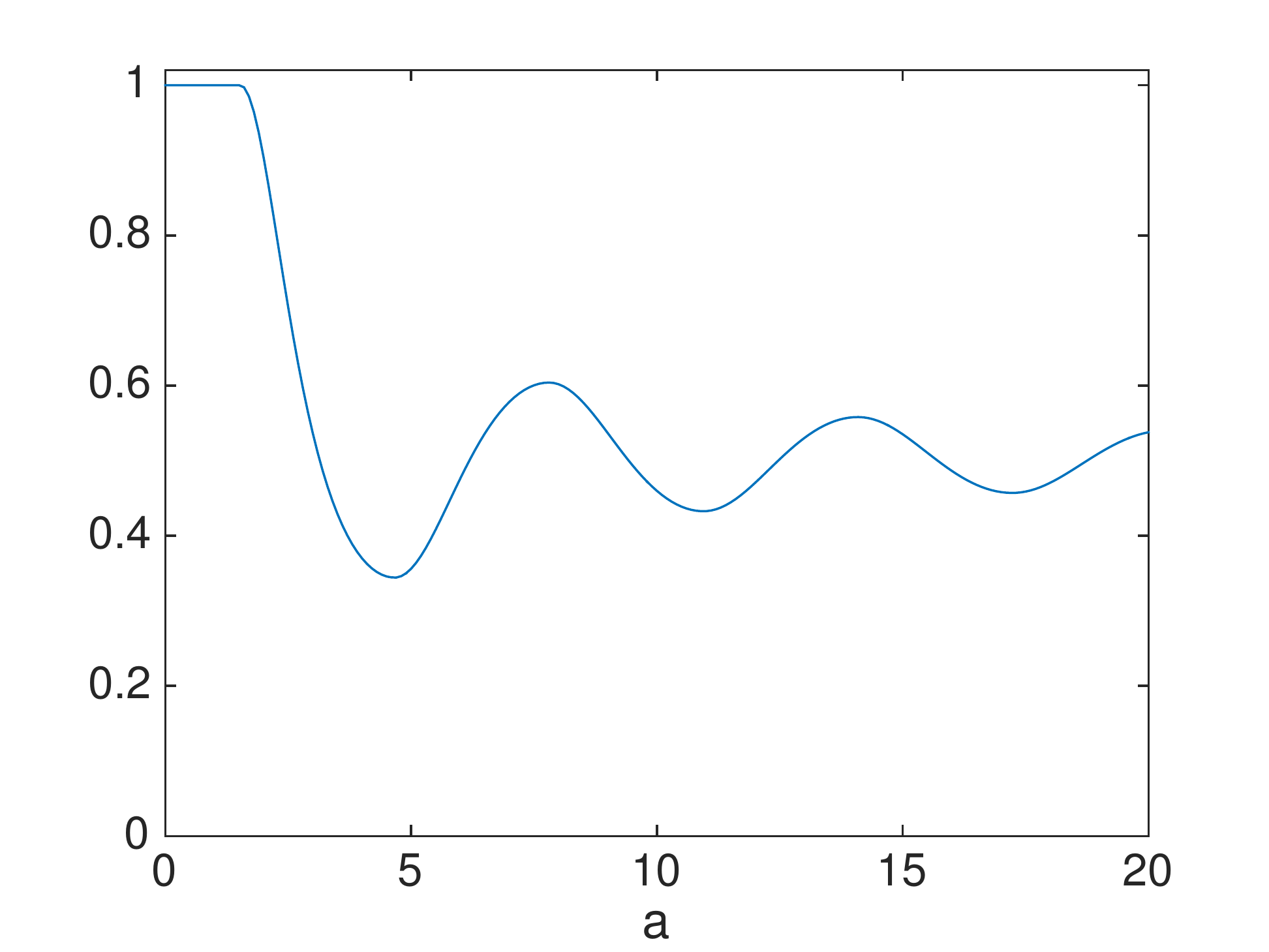}}
\caption{Convexity exploration of $h_{\cos}$ on the square $S_{0,0}(a)$ for various $a$ values.}
\label{graphs-0}
\end{figure}
Its Hessian is the diagonal matrix whose diagonal entries are the eigenvalues $\lambda_1(x,y)=\cos x $ and $\lambda_2(x,y)=\cos y $. Consequently, the pointwise convexity index is
\[
\mathrm{CONV}(h_{\cos},(x,y))={(\cos x )^{+}+(\cos y )^{+} \over |\cos x |+|\cos y |} ,
\]
and the $L_1$-based global index over the square $S_{0,0}(a)$ is
\begin{align*}
\mathrm{CONV}(a)
:=&\mathrm{CONV}(h_{\cos},S_{0,0}(a))
\\
=& {\int_{-a}^a \int_{-a}^a \big((\cos x )^{+}+(\cos y )^{+}\big)\mathrm{d}x\mathrm{d}y \over \int_{-a}^a \int_{-a}^a \big(|\cos x |+|\cos y |\big)\mathrm{d}x\mathrm{d}y  }
\\
=&{\int_0^a(\cos x )^{+}\mathrm{d}x \over \int_0^a |\cos x |\mathrm{d}x}.
\end{align*}
The pointwise and global indices are depicted in Figure~\ref{graphs-0}. The index $\mathrm{CONV}(a)$ is equal to $1$ for every $a \in (0, \pi/2]$, thus implying convexity of the function on the square $S_{0,0}(a)$ for the noted $a$ values, but when $a>\pi/2$, the amount of convexity decreases from $1$ (pure convexity) to a lower, though fluctuating, level of convexity, depending on the nature of the region being absorbed by the square $S_{0,0}(a)$ when the parameter $a$ grows. This concludes Example \ref{example-1}.
\end{example}

In the next section we show how the aforementioned indices arise from an minimization problem in the space of matrices. The argument concerns general symmetric  matrices, and it can readily be specialized to the aforementioned Hessian.

\section{Indices of positive semidefiniteness}
\label{results=2}

Let $H$ be a real symmetric $d\times d$ matrix. We rewrite it as
$H =Q \Lambda Q^{\top} $, where $Q $ is an orthogonal matrix and $\Lambda =\mathrm{diag}(\lambda_1 ,\dots , \lambda_d )$ is the diagonal matrix with $\lambda_i:= \lambda_i(H)$  denoting the eigenvalues of the matrix $H $. Define $\lambda_i^{+} =\max \{\lambda_i , 0\} $ and $\lambda_i^{-} =\max \{-\lambda_i , 0\} $. With the diagonal matrices $\Lambda^{+} :=\mathrm{diag}(\lambda^{+}_1 ,\dots , \lambda^{+}_d )$ and $\Lambda^{-} :=\mathrm{diag}(\lambda^{-}_1 ,\dots , \lambda^{-}_d )$, we obtain what we call the canonical decomposition
\begin{equation}\label{diag-1}
H =H^{+} -H^{-}
\end{equation}
of the matrix $H$, where $H^{+} =Q \Lambda^{+} Q^{\top} $ and $H^{-} =Q \Lambda^{-} Q^{\top} $. The matrices $H^{+} $ and $H^{-} $ are positive semidefinite. Note also that since the nuclear norm of the matrix $H$ is
\[
 \Vert H  \Vert_{*} =\sum_{i=1}^d |\lambda_i |,
\]
we have the equation $\Vert H  \Vert_{*}=\Vert H^{+}  \Vert_{*}+\Vert H^{-}  \Vert_{*}$.

Of course, there can be many decompositions of $H $ as the difference of two positive semidefinite matrices $H_1$ and $H_2$, but the canonical decomposition is a very special one, whose minimalist nature is elucidated in the following lemma.

\begin{lemma}\label{lemma-1}
If $H=H_1-H_2 $ for any pair of symmetric and positive semidefinite matrices $ H_1 $ and $ H_2 $, then $\mathrm{tr}(H^{+})\le \mathrm{tr}(H_1)$ and also $\mathrm{tr}(H^{-})\le \mathrm{tr}(H_2)$.
\end{lemma}

\begin{proof}
Since $H_1-H_2=H^{+}-H^{-}$, the two bounds in the formulation of Lemma \ref{lemma-1} are equivalent. Hence, it suffices to prove only one of them, say $\mathrm{tr}(H^{+})\le \mathrm{tr}(H_1)$, which is equivalent to
\begin{equation}\label{eigen-1}
\sum_{i=1}^d \lambda_i^{+}(H) \le \sum_{i=1}^d \lambda_i(H_1)
\end{equation}
because $\mathrm{tr}(H^{+})=\mathrm{tr}(Q\Lambda^{+}Q^{\top})=\mathrm{tr}(\Lambda^{+})$.
Next we write $\sum_{i=1}^d \lambda_i(H_1)=\sum_{i=1}^d \lambda_i^{\downarrow}(H_1)$, where $\lambda_1^{\downarrow}(H_1), \dots , \lambda_d^{\downarrow}(H_1)$ denote the eigenvalues $\lambda_1(H_1), \dots , \lambda_d(H_1)$ of the matrix $H_1$ arranged in the descending order, that is, $\lambda_1^{\downarrow}(H_1)\ge \dots \ge \lambda_d^{\downarrow}(H_1)$. Next we write $\lambda_i^{\downarrow}(H_1)=\lambda_i^{\downarrow}(H+H_2)$, and since $H$ is symmetric and $H_2$ is positive semidefinite, we have from Weyl's (1912) monotonicity theorem (e.g., Bhatia, 1997; Corollary III.2.3, p.~63) that, for every $i=1,\dots , d$, the bound  $\lambda_i^{\downarrow}(H+H_2)\ge \lambda_i^{\downarrow}(H)$ holds, which is of course equivalent to the bound
\begin{equation}\label{eigen-10}
\lambda_i^{\downarrow}(H_1)\ge \lambda_i^{\downarrow}(H).
\end{equation}
Since $H_1$ is positive semidefinite, all its eigenvalues $\lambda_1^{\downarrow}(H_1), \dots , \lambda_d^{\downarrow}(H_1)$ are non-negative. Since the matrix $H$ is just symmetric, its eigenvalues $\lambda_1^{\downarrow}(H), \dots , \lambda_d^{\downarrow}(H)$ may or may not be non-negative. Hence, from bound (\ref{eigen-10}) we have
\[
\lambda_i^{\downarrow}(H_1)\ge \max\{\lambda_i^{\downarrow}(H),0\},
\]
which holds for all $i=1,\dots , d$. Consequently, we have
\begin{align*}
\sum_{i=1}^d \lambda_i(H_1)=\sum_{i=1}^d \lambda_i^{\downarrow}(H_1)
&\ge \sum_{i=1}^d \max\{\lambda_i^{\downarrow}(H),0\}
\\
&= \sum_{i=1}^d \max\{\lambda_i(H),0\}
=\sum_{i=1}^d \lambda_i^{+}(H),
\end{align*}
which establishes bound (\ref{eigen-1}). The proof of Lemma \ref{lemma-1} is finished.
\end{proof}

The above lemma plays a pivotal role when establishing a geometric interpretation of the index $\mathrm{LOC}(h,\mathbf{x})$:  when specialized to the Hessian $H(\mathbf{x})$, it follows from the next theorem that $\mathrm{LOC}(h,\mathbf{x})$ is the minimal nuclear-distance of the Hessian from the space $\mathcal{M}^{+}$ of all symmetric, positive semidefinite, $d\times d$ matrices.

\begin{theorem}\label{def-1}
For every  symmetric, $d\times d$ matrix $H$, we have
\begin{equation}\label{def-2}
\inf_{M\in \mathcal{M}^{+}}\Vert H -M \Vert_{*} =\Vert H^{-} \Vert_{*}.
\end{equation}
Consequently, the following quantities
\begin{align*}
\mathrm{LOPS}(H)&=\Vert H^{-} \Vert_{*},
\\
\mathrm{NLOPS}(H)&={\Vert H^{-} \Vert_{*} \over \Vert H \Vert_{*}},
\\
\mathrm{PS}(H)&={\Vert H^{+} \Vert_{*} \over \Vert H \Vert_{*}}
\end{align*}
define, respectively, the index of lack of positive semidefiniteness (LOPS), its normalized version (NLOPS), and the index of positive semidefiniteness (PS).
\end{theorem}

\begin{proof}
We begin with the observation that since the space $\mathcal{M}^{+}$ is finite-dimensional and closed, we can find an element $D\in \mathcal{M}^{+}$ such that
\begin{equation}\label{loc-2}
\inf_{M\in \mathcal{M}^{+}}\Vert H-M \Vert_{*} = \Vert H-D \Vert_{*}.
\end{equation}
Denote $E:=H-D$ and write its canonical decomposition $E=E^{+}-E^{-}$. Since $H=D+E$, we have $H=D+E^{+}-E^{-}$ and thus $H-(D+E^{+})=-E^{-}$. Due to the latter equation, we have
\begin{align}
\Vert H-(D+E^{+}) \Vert_{*}
&= \Vert E^{-}\Vert_{*}
\notag
\\
&\le \Vert E \Vert_{*}
= \inf_{M\in \mathcal{M}^{+}}\Vert H-M \Vert_{*},
\label{eq-10}
\end{align}
where the right-most equation holds because $E=H-D$. Since $D+E^{+} \in \mathcal{M}^{+}$, inequality (\ref{eq-10}) turns into equality, and thus we have $\Vert E^{-}\Vert_{*}= \Vert E \Vert_{*} $. This equation implies $\Vert E^{+}\Vert_{*}=0 $ because $\Vert E \Vert_{*}=\Vert E^{+} \Vert_{*}+\Vert E^{-} \Vert_{*}$, and we therefore conclude that $E^{+}=0$. Consequently, we have $E=-E^{-}$, which in turn implies $-E\in \mathcal{M}^{+}$ because $E^{-}\in \mathcal{M}^{+}$. Since  $D\in \mathcal{M}^{+}$ and $-E\in \mathcal{M}^{+}$, the equation $H=D+E$ written as $H=D-(-E)$ gives us an alternative (which may or may not be canonical) decomposition of $H$ into the difference of two positive semidefinite, symmetric, $d\times d$ matrices.

Hence, in summary, we have two decompositions of the matrix $H$ into the difference of positive semidefinite, symmetric, $d\times d$ matrices: the \textit{canonical} decomposition $ H=H^{+}-H^{-} $ and the \textit{alternative}  decomposition $ H=H_1-H_2 $ with $ H_1, H_2 \in \mathcal{M}^{+}$. We have $\mathrm{tr}(H^{-})=\Vert H^{-} \Vert_{*} $. Next, upon recalling that $H_2=-E$, we have $\mathrm{tr}(H_2)=\mathrm{tr}(-E)=\Vert H-D \Vert_{*} $. From these observations and the bound $\mathrm{tr}(H^{-})\le \mathrm{tr}(H_2)$ (Lemma \ref{lemma-1}), we conclude that
\begin{equation}\label{eq-10a}
\Vert H^{-} \Vert_{*}\le \Vert H-D \Vert_{*} .
\end{equation}
But $H^{-} =H-H^{+} $ and also $H^{+}\in \mathcal{M}^{+}$. Hence, inequality (\ref{eq-10a}) turns into equality, which in turn implies
\begin{equation}\label{eq-11}
\Vert H-D \Vert_{*} =\Vert H-H^{+} \Vert_{*} .
\end{equation}
Since $\Vert H-H^{+} \Vert_{*} $ is equal to $ \Vert H^{-} \Vert_{*}$, equation (\ref{eq-11}) implies statement (\ref{def-2}) and in this way concludes the proof of Theorem \ref{def-1}.
\end{proof}

We are now in the position to give additional insight into Definitions \ref{def-0a}--\ref{def-0c} by deriving alternative representations of the indices of (lack of) convexity introduced in Section~\ref{results=1}.

\begin{corollary}\label{cor-1}
For every $h\in C^2(G)$, we have the equations
\begin{gather*}
\mathrm{LOC}(h,\mathbf{x})=\Vert H_h^{-}(\mathbf{x}) \Vert_{*},
\\
\mathrm{NLOC}(h,\mathbf{x})={\Vert H_h^{-}(\mathbf{x}) \Vert_{*} \over \Vert H_h(\mathbf{x}) \Vert_{*}},
\\
\mathrm{CONV}(h,\mathbf{x})={\Vert H_h^{+}(\mathbf{x}) \Vert_{*} \over \Vert H_h(\mathbf{x}) \Vert_{*}},
\end{gather*}
which are equivalent to equations (\ref{loc-1}), (\ref{nloc-1}) and (\ref{c-1}), respectively.
\end{corollary}

\section{An illustration}
\label{illustration}

The notions and quantities that we shall use in this section are standard in the literature on risk measurement and management (e.g., McNeil, Frey and Embrechts,~2015). Nevertheless, to make the text more self-contained and readable, we shall recall some of the standard definitions, such as those of the value at risk, expected shortfall, and a few other ones.

\subsection{Description}
\label{problem}

Consider a company with $d$ business lines, $i=1,\dots , d$. Running business is costly and risky. For each $i=1,\dots , d$, let $x_i$ be the (non-random) amount of risk capital allocated to the $i^{\textrm{th}}$ business line in order to cover its loss $X_i$, which is not known beforehand and thus a random variable. Denote the cumulative distribution function (cdf) of $X_i$ by $F_i$.

The expected loss not covered by the allocated capital $x_i$ is $\mathbf{E}[(X_i-x_i)_{+}]$, where $(X_i-x_i)_{+}$ is $X_i-x_i$ when $X_i>x_i$, and $0$ otherwise. Obviously, if $x_i$ is very large, then $\mathbf{E}[(X_i-x_i)_{+}]$ is very small, but keeping capital `frozen' is costly. Let $\ell_i(x_i)$ be the loss, or penalty, associated with the $i^{\textrm{th}}$ business line for keeping the capital $x_i$. Hence, the total loss associated with the $i^{\textrm{th}}$ business line is
\[
h_i(x_i)=\mathbf{E}[(X_i-x_i)_{+}]+\ell_i(x_i).
\]
It is natural to assume that $\ell_i$ is a non-decreasing function taking non-negative values. For example, let
\begin{equation}\label{loss-1}
\ell_i(x_i)=(1-p_i)x_i
\end{equation}
for some $p_i\in (0,1)$. Referring to insurance, a practically relevant value would be $p_i=0.99$. The function $h_i$ achieves its minimum at the point $x_i^0:=\mathrm{VaR}_{p_i}(X_i)$, where $\mathrm{VaR}_{p_i}(X_i)$ is the value-at-risk at the $p_i^{\textrm{th}}$ quantile of the underlying loss $X_i$. That is,
\[
\mathrm{VaR}_{p_i}(X_i)=\inf \{ x\in \mathbb{R}~:~ F_i(x)\ge p_i \},
\]
which, in the insurance lingo,  can be interpreted as the smallest premium that needs to be charged in order to cover at least $p_i\times 100\%$ of losses. It is well known (e.g., McNeil et al., 2015) that $h_i(x_i^0)$ is equal to $(1-p_i)\mathrm{AVaR}_{p_i}(X_i)$, where $\mathrm{AVaR}_{p_i}(X_i)$ is the average-value-at-risk at the $p_i^{\textrm{th}}$ quantile of the underlying loss $X_i$, that is,
\[
\mathrm{AVaR}_{p_i}(X_i)= {1\over 1-p_i} \int_{p_i}^{\infty }\mathrm{VaR}_{t}(X_i)\mathrm{d}t.
\]
This is one of the fundamental risk measures currently in use in insurance and banking.

When $\ell_i(x_i)$ is given by equation (\ref{loss-1}), the function $h_i$ is convex, but other forms of the loss function $\ell_i$ are also of much interest. For example, let
\begin{equation}\label{loss-2}
\ell_i(x_i)=(1-p_i)x_i^{\alpha_i}
\end{equation}
for some $\alpha_i>0$. Under this loss function, the function $h_i$ may or may not be convex over its entire domain of definition, depending on the value of $\alpha_i$, but it can nevertheless be convex in some regions of its domain of definition. It therefore becomes natural to specify those regions where the function $h_i$ is convex; and where it is not, we would then wish to assess the extent of its non-convexity. The indices introduced in Section \ref{results=1} provide much needed answers and insights into such issues, and we shall illustrate this in a moment.

\subsection{Parameter choices}
\label{choices}

The losses emanating from the individual business lines need to be aggregated onto the company's level, which can successfully be implemented with the help of the weighted generalized mean
\begin{equation}\label{hhh-2}
h(\mathbf{x})=\bigg ( \sum_{i=1}^d w_i h_i^{\beta}(x_i) \bigg )^{1/\beta }
\end{equation}
for appropriately chosen values of the parameter $\beta \in \mathbb{R}$ and weights $w_i\ge 0$, which have to be such that $\sum_{i=1}^d w_i=1$. For example, the parameter $\beta$ could be equal to $-1$, $0$, $1$ or $2$, which give rise to the harmonic, geometric, arithmetic, and quadratic means, respectively. These are also the choices that we adopt in our following numerical explorations.

For the sake of expository simplicity, we deal with only two business lines, $d=2$, and set the two weights $w_1$ and $w_2$ to $1/2$, which in practical terms means that the two business lines are viewed as being of the same importance within the company. Furthermore, let the two loss functions $\ell_i$ be those defined by equation (\ref{loss-2}) with the parameter choices $p_i=0.99$, which corresponds to the $99^{\textrm{th}}$ risk percentile, and $\alpha_i=1/4$ for both $i=1,2$.
Since the expected shortfall $\mathbf{E}[(X_i-x_i)_{+}]$ can be written as $\int_{x_i}^{\infty }(1-F_i(t))\mathrm{d}t$, function (\ref{hhh-2}) becomes
\begin{align*}
h(x,y)&=\bigg ( {1\over 2}\bigg ( \int_x^{\infty }(1-F_1(t))\mathrm{d}t + (1-p_1)x^{\alpha_1}\bigg )^{\beta}
+{1\over 2}\bigg ( \int_x^{\infty }(1-F_2(t))\mathrm{d}t + (1-p_2)y^{\alpha_2}\bigg )^{\beta} \bigg )^{1/\beta}
\\
&= \bigg ( {1\over 2}\bigg ( \int_x^{\infty }(1-F_1(t))\mathrm{d}t + 0.01 x^{1/4}\bigg )^{\beta}
+{1\over 2}\bigg ( \int_x^{\infty }(1-F_2(t))\mathrm{d}t + 0.01 y^{1/4}\bigg )^{\beta} \bigg )^{1/\beta}.
\end{align*}
For the sake of simplicity and thus transparency, let $X_1$ and $X_2$ follow the uniform on $[0,1]$ distribution, which means that their cdf's  $F_1(t)$ and $F_2(t)$ are equal to $t$ on the unit interval $[0,1]$. (We can interpret $X_1$ and $X_2$ as percentages.) This turns the function $h$ into
\begin{equation}\label{func-graphing}
h_{\beta }(x,y):=\bigg ( {1\over 2} g^{\beta}(x) + {1\over 2} g^{\beta}(y) \bigg )^{1/\beta}
\end{equation}
defined on the unit square $[0,1]\times [0,1]$, where
$g(z)= 0.5(1-z)^2 +0.01 z^{1/4}$ for all $ 0\le z\le 1 $.
We wish to explore the regions of convexity as well as of non-convexity of the function $h_{\beta }$, and to also quantify the extent of non-convexity when the function is not convex.

In Figures \ref{fig-41} and \ref{fig-42} we plot the function $h_{\beta }$ for the aforementioned four parameter $\beta $ values, that is, for $\beta=-1$, $0$, $1$ and $2$.
\begin{figure}[h!]
\centering
\subfigure[Function $h_{\beta }(x,y)$ when $\beta=-1$]{\includegraphics[width=0.42\textwidth]{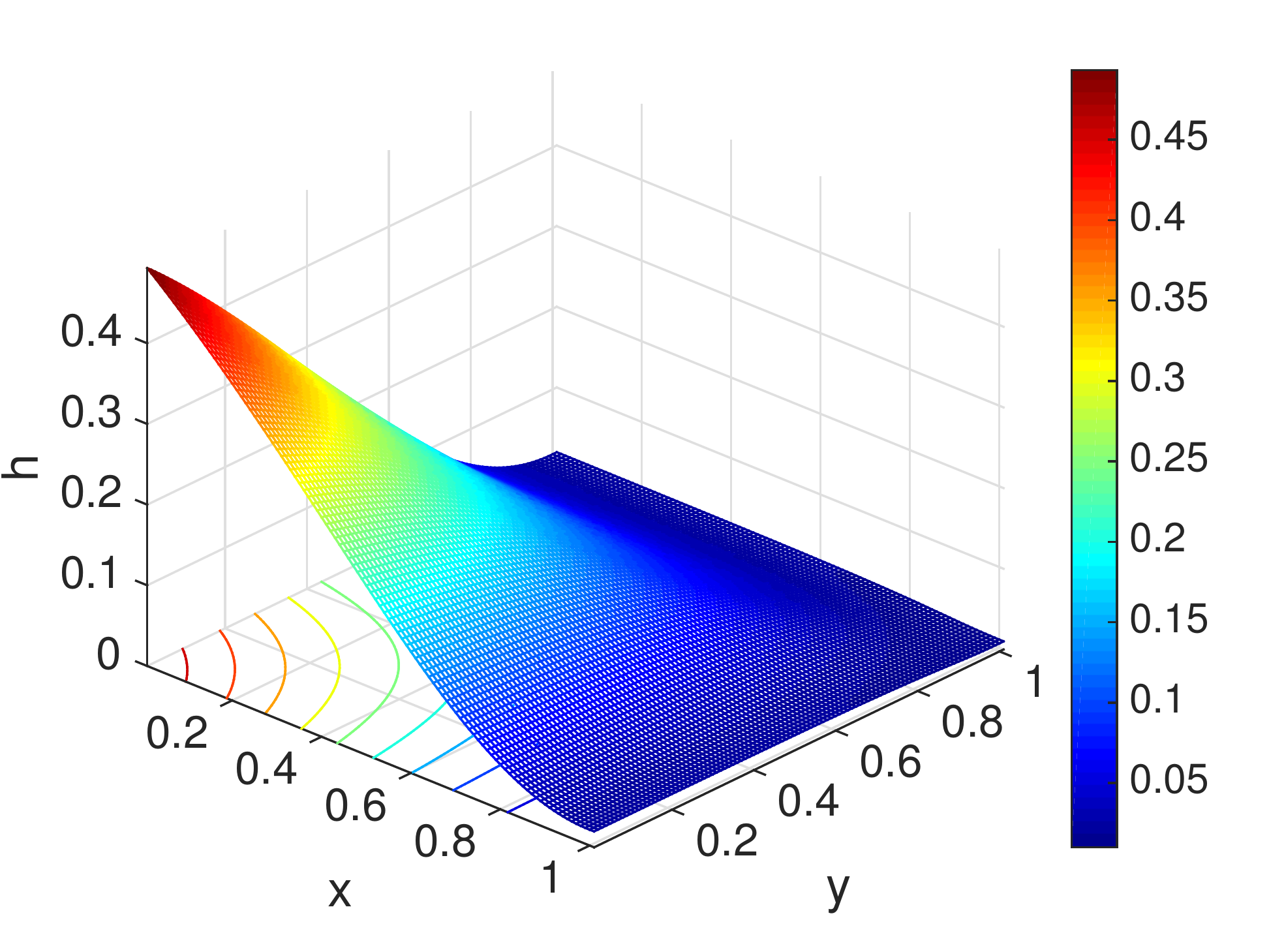}}\qquad
\subfigure[Function $h_{\beta }(x,y)$ when $\beta=0.001$]{\includegraphics[width=0.42\textwidth]{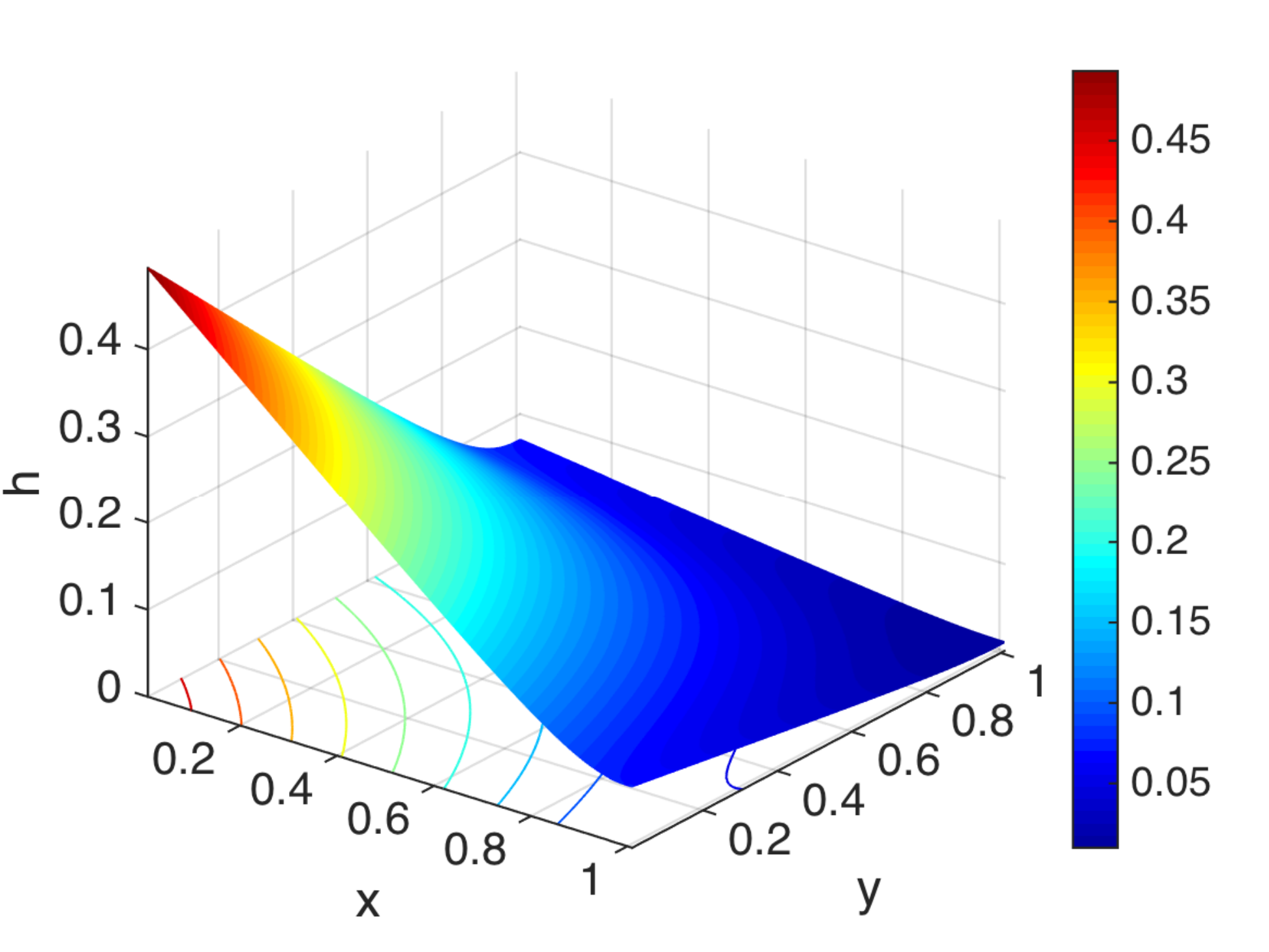}}\\\vspace*{8mm}
\subfigure[Global index $\mathrm{CONV}(a)$ when $\beta=-1$]{\includegraphics[width=0.42\textwidth]{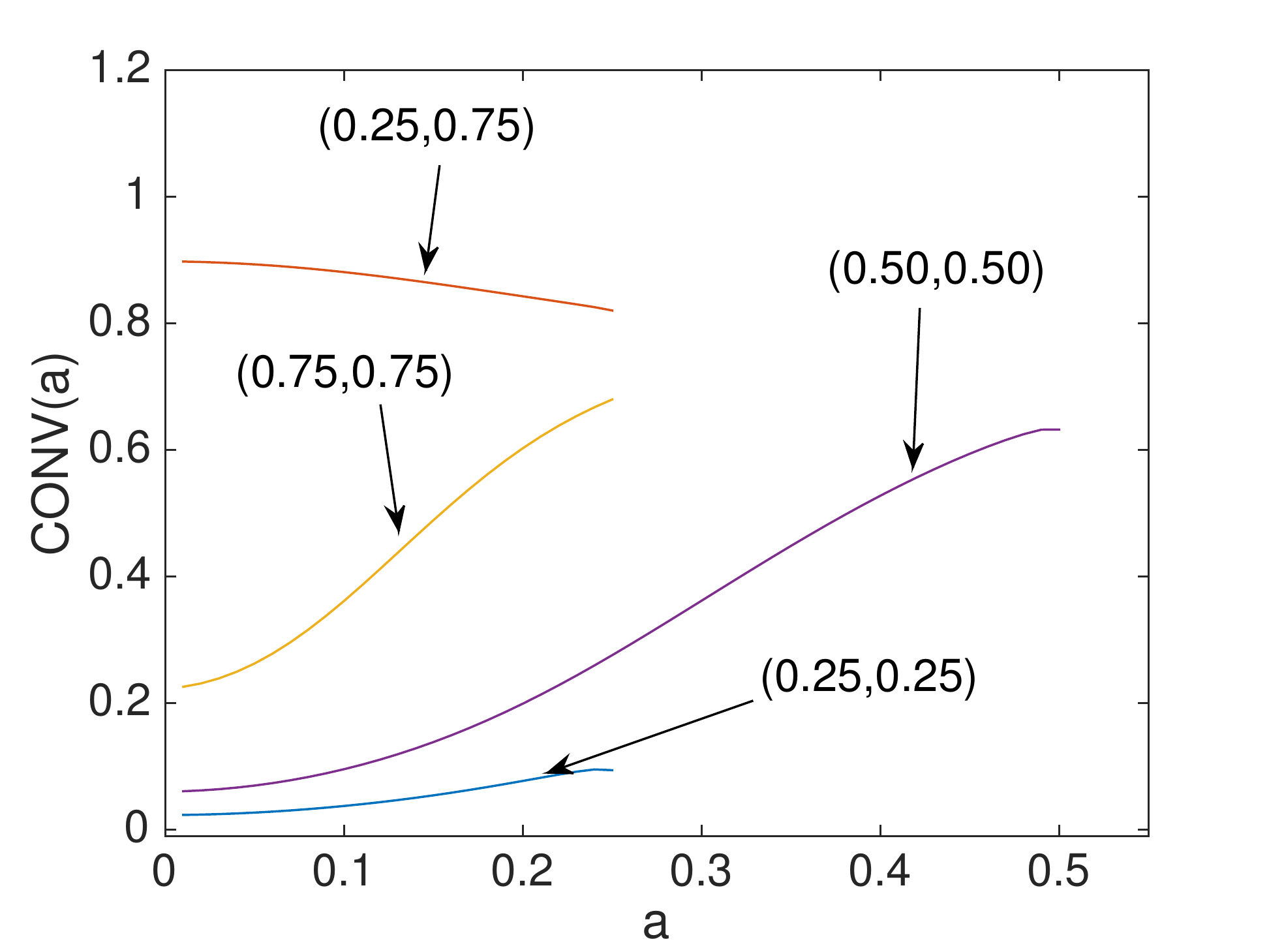}}\qquad
\subfigure[Global index $\mathrm{CONV}(a)$ when $\beta=0.001$]{\includegraphics[width=0.42\textwidth]{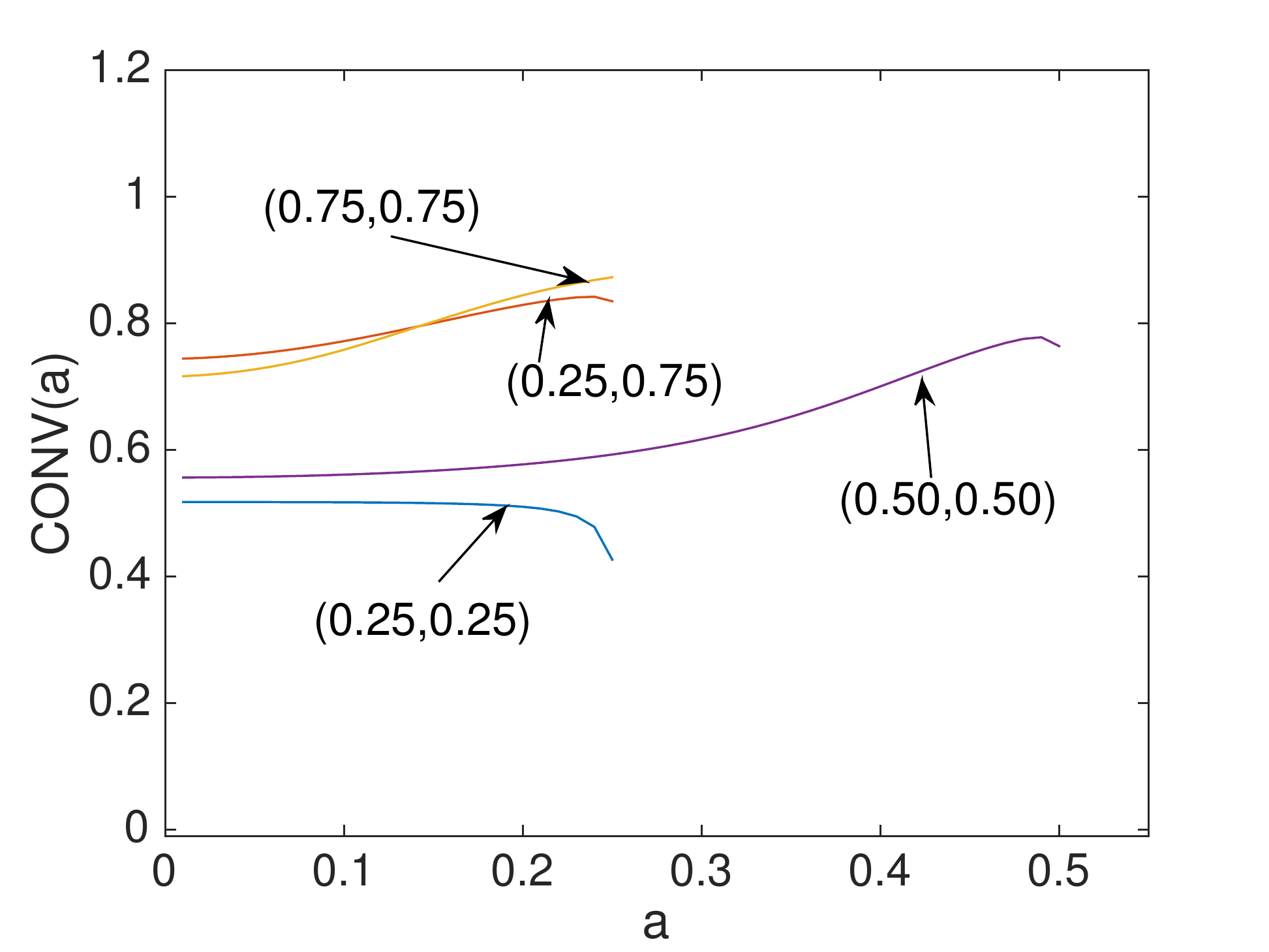}}
\caption{The function $h_{\beta }(x,y)$ and its index $\mathrm{CONV}(a)$ for $\beta =-1$ (harmonic mean) and $\beta =0.001$ (nearly geometric mean).}
\label{fig-41}
\end{figure}
Though it may look simple, the function $h_{\beta }$ is not simple enough to easily obtain  closed-form formulas for the two eigenvalues of the Hessian, unlike we did in Example \ref{example-1} above. Hence, we resort to numerical methods, which are of primary practical interest anyway, and in addition to the surfaces themselves, in Figures~\ref{fig-41} and \ref{fig-42} we also depict the global convexity index
\begin{equation}
\mathrm{CONV}(a):=\mathrm{CONV}(h_{\beta },S_{x_0,y_0}(a))
= { \int_{x_0-a}^{x_0+a}\int_{y_0-a}^{y_0+a}\big(\lambda_1^{+}(x,y)+\lambda_2^{+}(x,y)\big) dxdy
\over
\int_{x_0-a}^{x_0+a}\int_{y_0-a}^{y_0+a}\big( |\lambda_1(x,y)|+|\lambda_2(x,y)|\big) dxdy }
\label{ind-graphing}
\end{equation}
over the square
\[
S_{x_0,y_0}(a)=(x_0,y_0)+(-a,a)\times (-a,a)
\]
and under the parameter specifications
\begin{itemize}
  \item $(x_0,y_0)=(0.25,0.25)$ and $0<a<0.25$
  \item $(x_0,y_0)=(0.25,0.75)$ and $0<a<0.25$
  \item $(x_0,y_0)=(0.75,0.75)$ and $0<a<0.25$
  \item $(x_0,y_0)=(0.50,0.50)$ and $0<a<0.5$
\end{itemize}
The difference between the ranges of $a$ in the above specifications is due to the need to keep the square
$S_{x_0,y_0}(a)$ inside the domain of definition $[0,1]\times [0,1]$ of the function $h_{\beta }$.
\begin{figure}[h!]
\centering
\subfigure[Function $h_{\beta }(x,y)$ when $\beta=1$]{\includegraphics[width=0.42\textwidth]{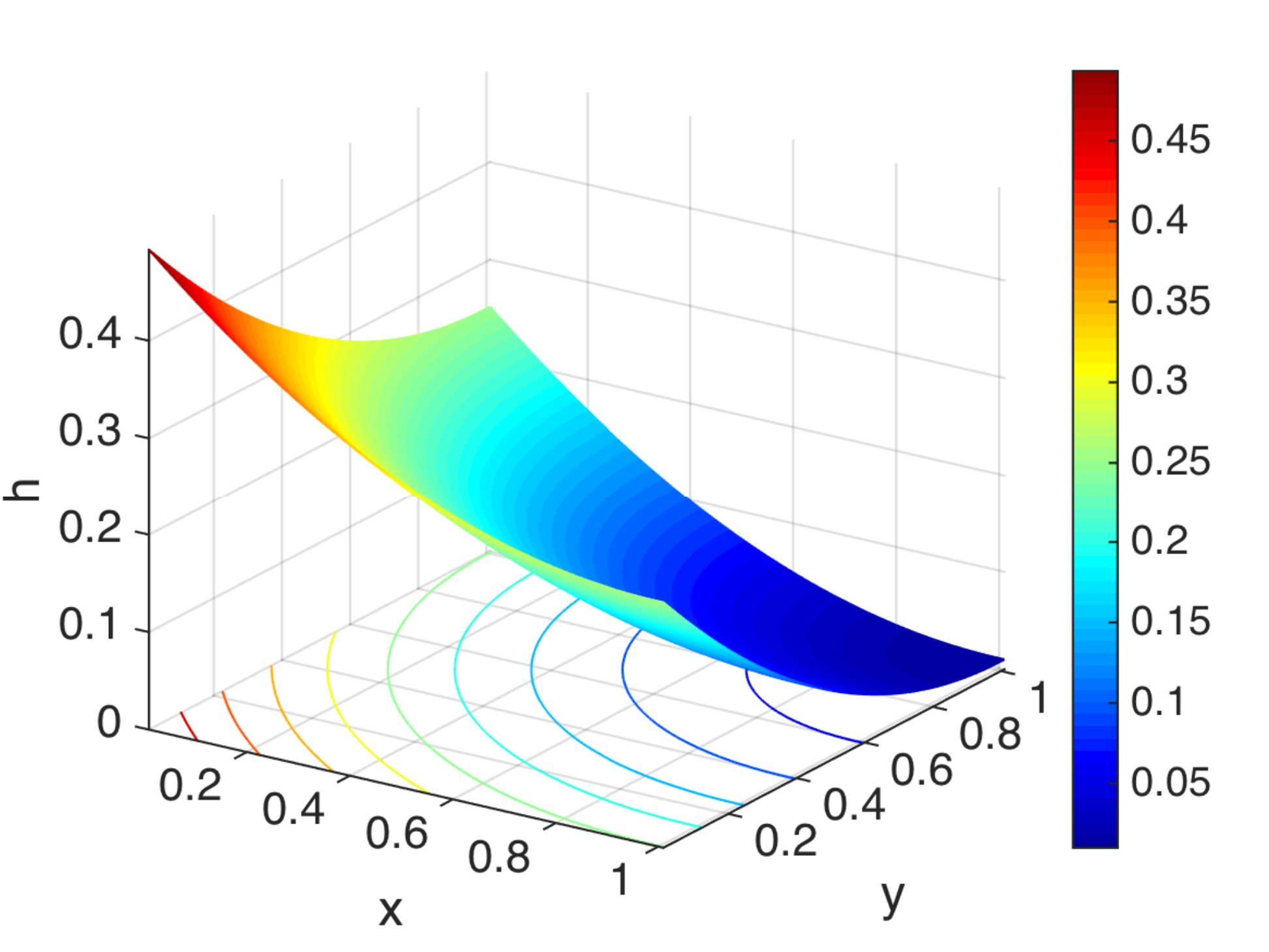}}\qquad
\subfigure[Function $h_{\beta }(x,y)$ when $\beta=2$]{\includegraphics[width=0.42\textwidth]{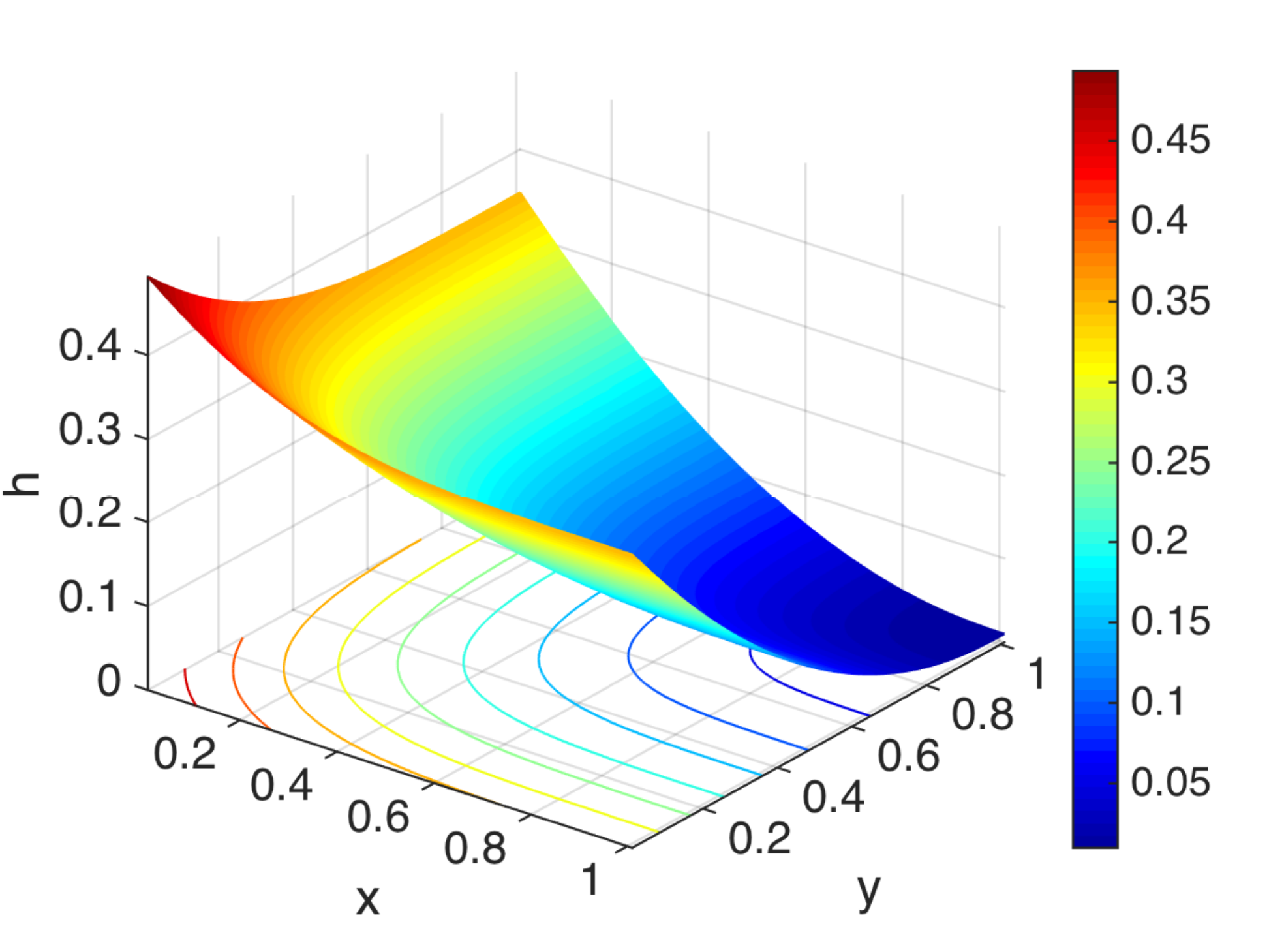}}\\\vspace*{8mm}
\subfigure[Global index $\mathrm{CONV}(a)$ when $\beta=1$]{\includegraphics[width=0.42\textwidth]{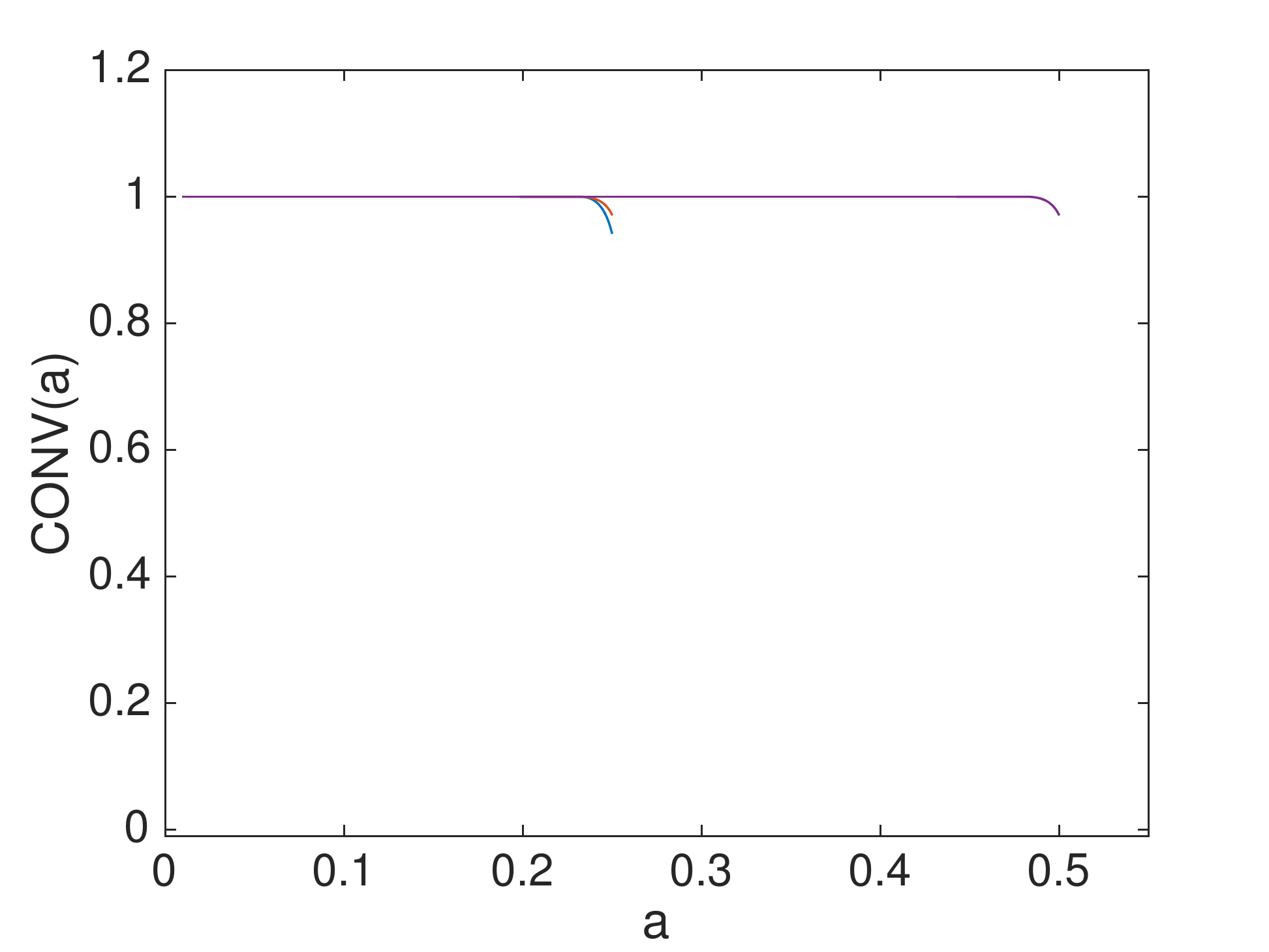}}\qquad
\subfigure[Global index $\mathrm{CONV}(a)$ when $\beta=2$]{\includegraphics[width=0.42\textwidth]{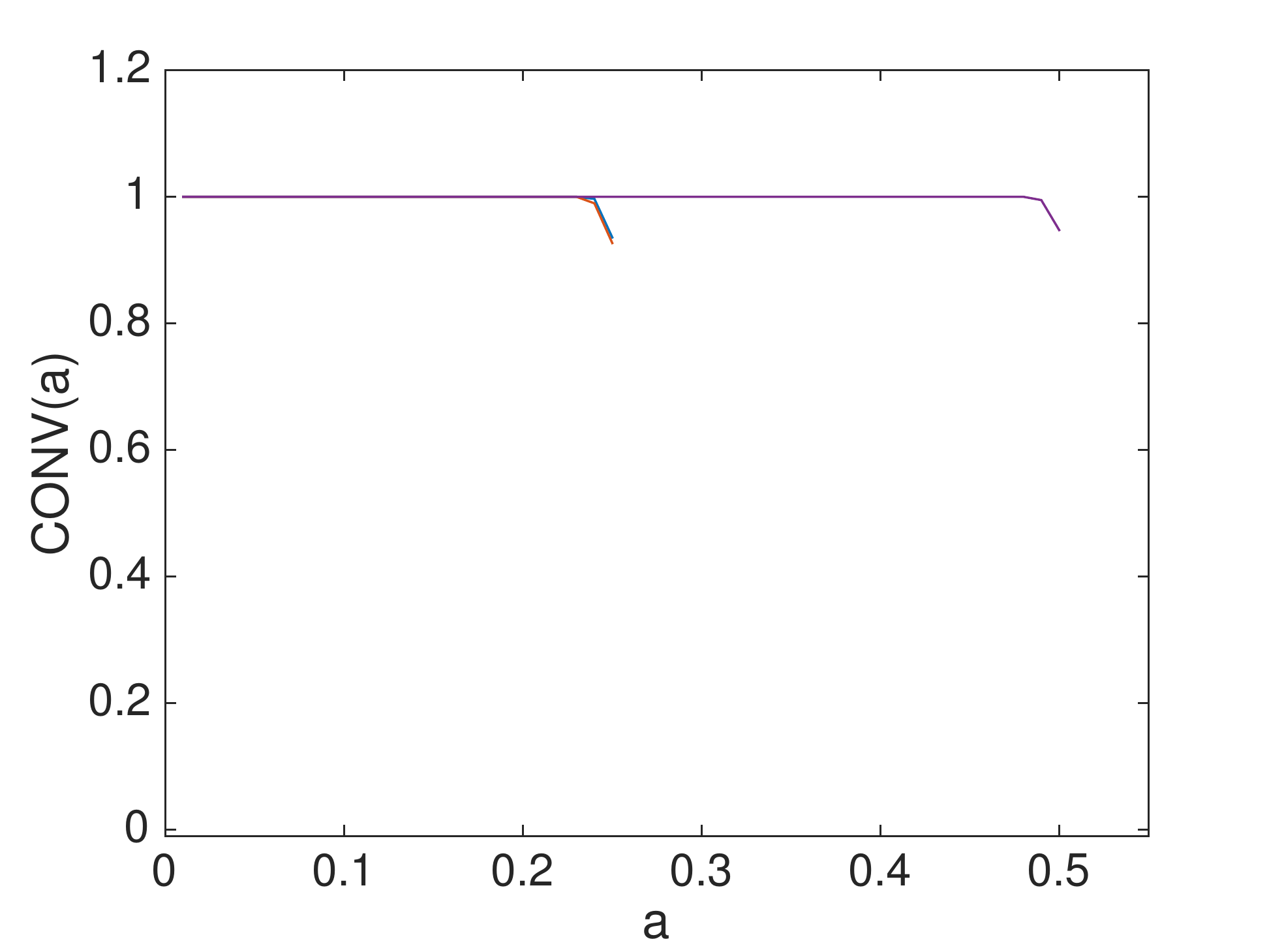}}
\caption{The function $h_{\beta }(x,y)$ and its index $\mathrm{CONV}(a)$ for $\beta =1$ (arithmetic mean) and $\beta =2$ (quadratic mean).}
\label{fig-42}
\end{figure}

\subsection{Findings}
\label{findings}

We now comment on Figures \ref{fig-41} and \ref{fig-42}, and start with the two  left-hand panels of Figure~\ref{fig-41} that correspond to the case $\beta=-1$ of the harmonic mean. Note that the three points $(x_0,y_0)=(0.25,0.25)$, $(0.5,0.5)$, and $(0.75,0.75)$ are on the ridgeline of the surface, and thus result in low starting values of the global index $\mathrm{CONV}(a)$. Nevertheless, the index increases when the values of $a$ increase, due to the fact that larger convex regions on both sides of the ridgeline are absorbed into the index by the expanding square $S_{x_0,y_0}(a)$. Note that the functions $\mathrm{CONV}(a)$ corresponding to the aforementioned three points $(x_0,y_0)$ dominate each other, which is a consequence of the fact that the points are located in places of decreasing sharpness  of the ridgeline. Note also that the function $\mathrm{CONV}(a)$ that corresponds to the point $(x_0,y_0)=(0.25,0.75)$ is quite high, due to the fact that this point is far away from the ridgeline, but since the square $S_{x_0,y_0}(a)$ approaches the ridgeline when $a$ increases from $0$ to $0.25$, the function tends to decrease due to the absorption of increasingly larger non-convex regions.

We now look at the two right-hand panels of Figure~\ref{fig-41} that correspond to the case $\beta=0.001$. If compared to the previous case $\beta=-1$, the surface $h_{\beta }$, though still somewhat concave, it nevertheless has lost its pronounced ridgeline and thus a part of its concavity. This is reflected by higher positioned functions $\mathrm{CONV}(a)$ in right-hand panel (d) of Figure~\ref{fig-41} than those in left-hand panel (c) of the same figure. Another feature of the function $\mathrm{CONV}(a)$ when $\beta=0.001$ is that it drops noticeably at the end of its domain of definition, which in the case of the point $(0.25,0.25)$ would be near $a=0.25$.

The functions $\mathrm{CONV}(a)$ depicted in panels (c) and (d) of Figure~\ref{fig-42} are  virtually equal to $1$, except near the end-points $a=0.25$ and $a=0.5$ of the respective  domains of definition, which depend on the point $(x_0,y_0)$. This suggests virtually total convexity of the surfaces $h_{\beta }$ corresponding to $\beta =1$ (arithmetic mean) and $\beta =2$ (quadratic mean). Nevertheless, the functions are dropping at the end of their respective domains of definition, which suggests some loss of convexity near the edges of the unit square $[0,1]\times [0,1]$. To understand the extent of this loss, we magnify the functions near $a=0.25$ and depict them in Figure~\ref{fig-43}.
\begin{figure}[h!]
\centering
\subfigure[When $\beta=1$]{\includegraphics[width=0.41\textwidth]{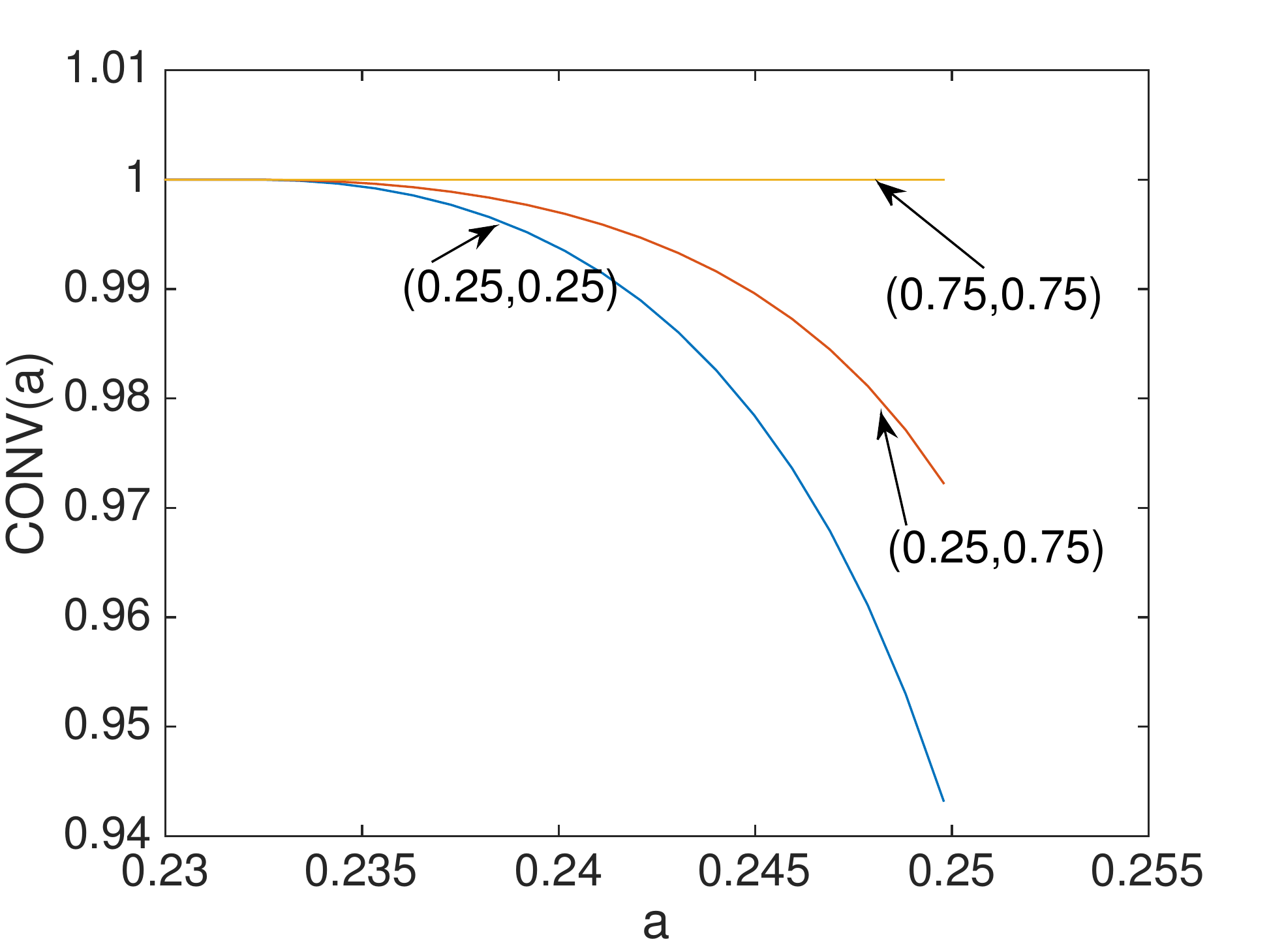}}\qquad
\subfigure[When $\beta=2$]{\includegraphics[width=0.41\textwidth]{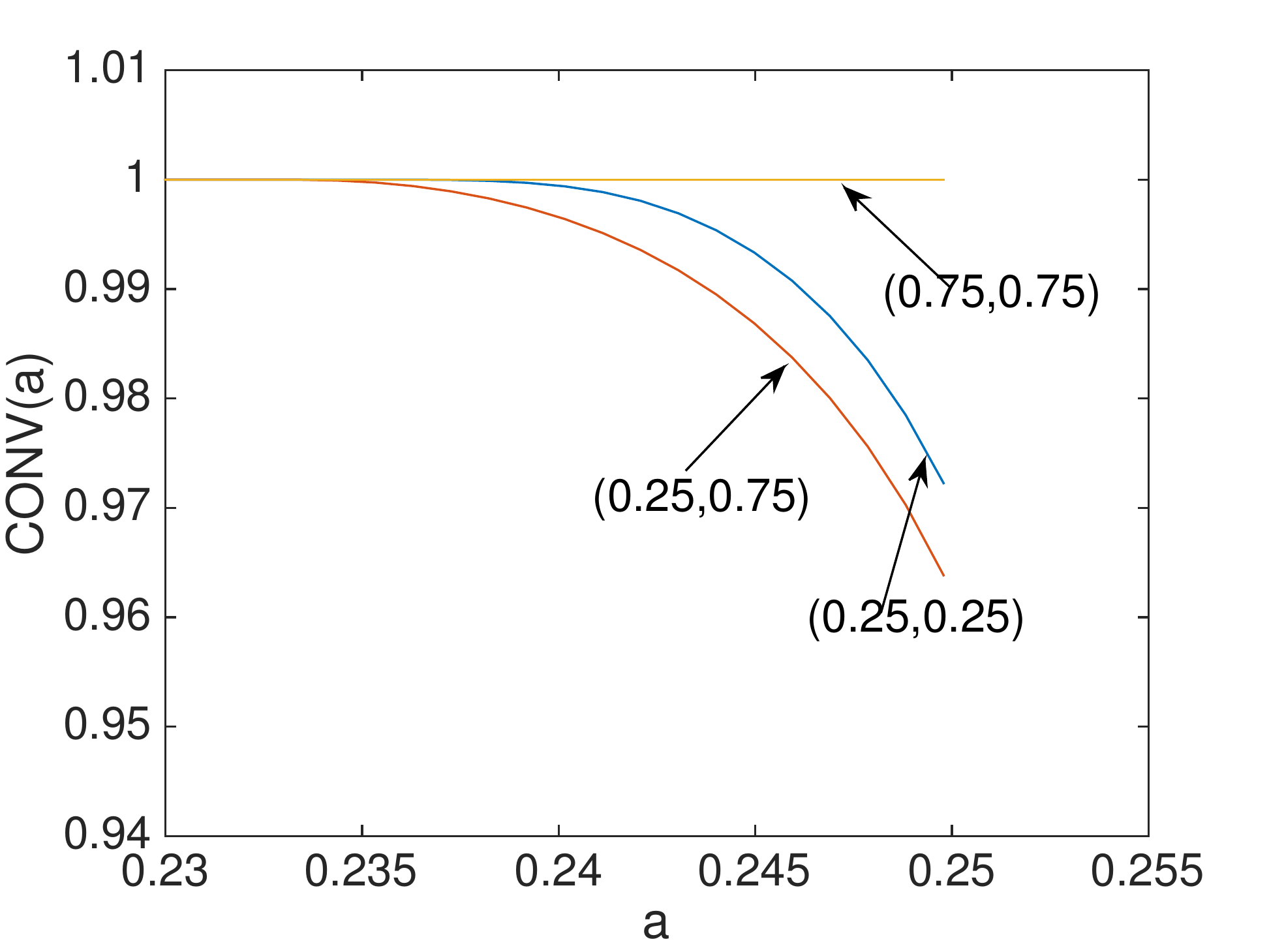}}\qquad
\caption{Magnified graphs of the convexity index $\mathrm{CONV}(a)$.}
\label{fig-43}
\end{figure}
To see where the loss of convexity happens on the surface $h_{\beta }$, we visualize small portions of it in Figure~\ref{fig-44} near the vertices $(0,0)$ and $(0,1)$ under the parameter choice $\beta=1$. (When $\beta=2$, the graphs are virtually the same and we therefore do not present them.) Note also from Figure~\ref{fig-43} that, as suggested by higher values of $\mathrm{CONV}(a)$ in right-hand panel (b) than in left-hand panel (a), the loss of convexity near the edges of the domain of definition $[0,1]\times [0,1]$ is smaller when $\beta=2$ than when $\beta=1$.
\begin{figure}[h!]
\centering
\subfigure[Near the vertex $(0,0$)]{\includegraphics[width=0.42\textwidth]{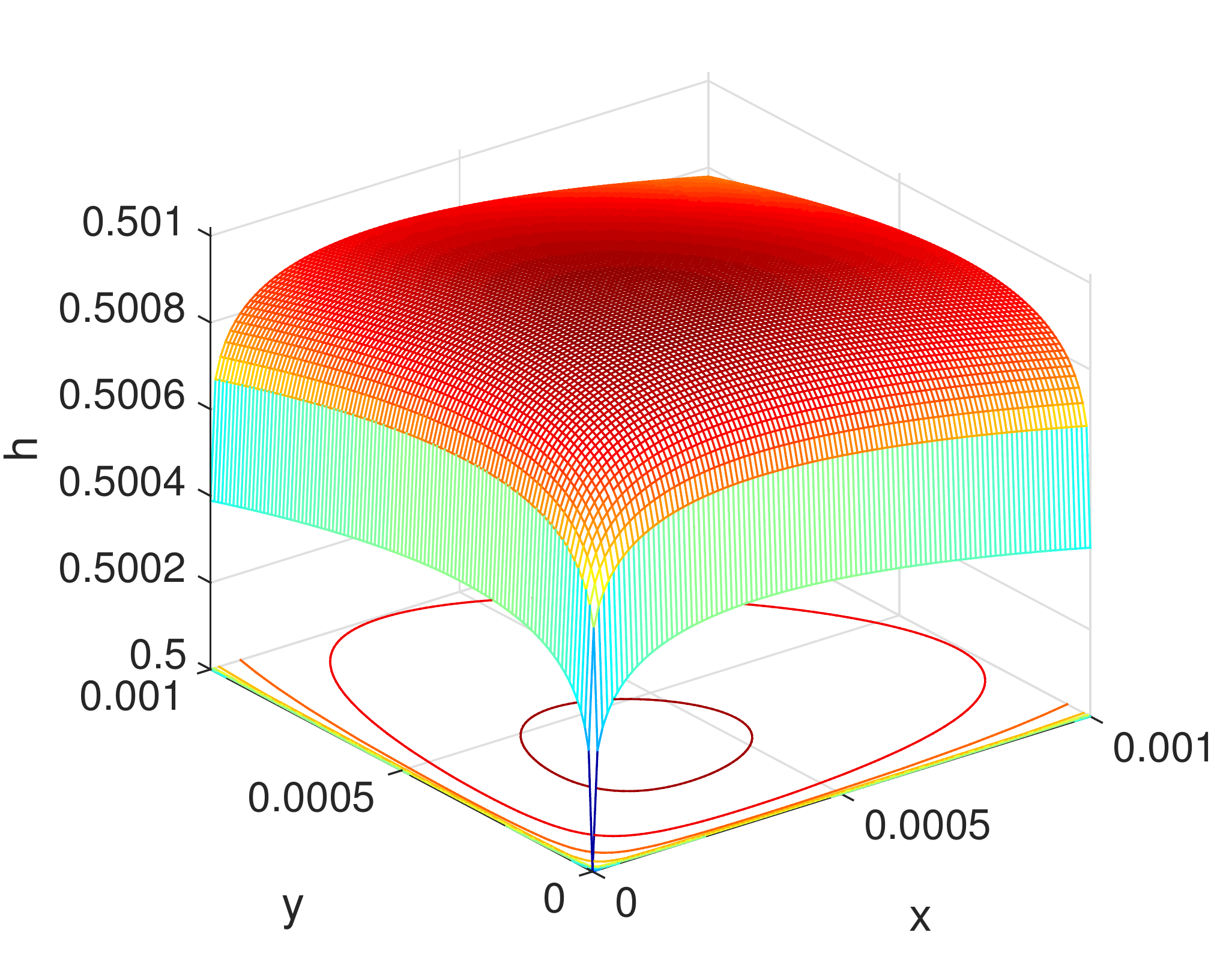}}\qquad
\subfigure[Near the vertex $(0,1$)]{\includegraphics[width=0.42\textwidth]{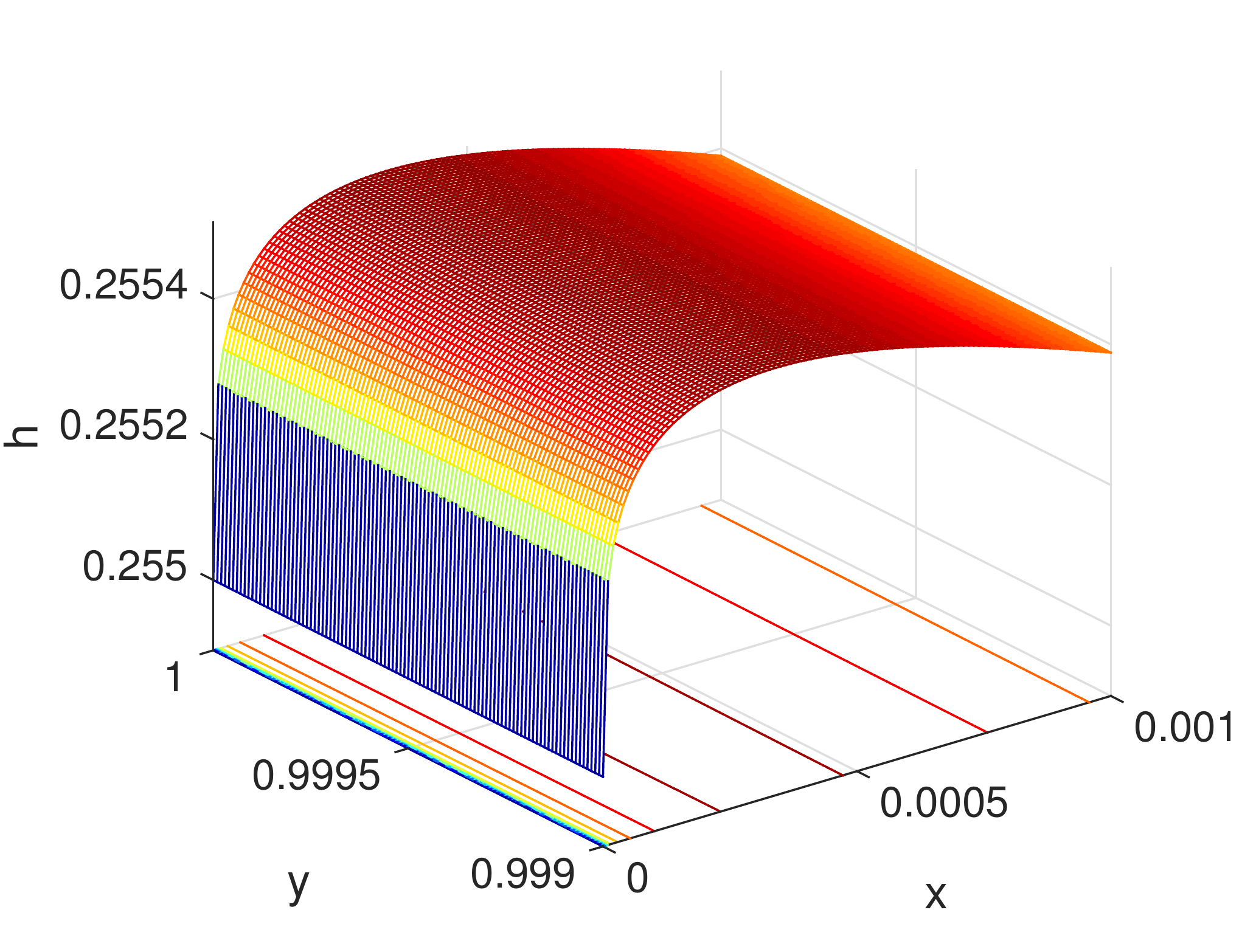}}\vspace*{8mm}
\caption{The magnified function $h_{\beta }(x,y)$ for $\beta=1$.}
\label{fig-44}
\end{figure}

\section{Concluding notes}
\label{conclude}

Theorem \ref{def-1} is the backbone of this paper. It introduces an index for measuring  lack of positive semidefiniteness in symmetric matrices. When specialized to the Hessian of sufficiently smooth functions, the theorem gives rise to an index suitable for measuring lack of convexity in functions, as discussed in Section \ref{results=1}. We have extensively illustrated these theoretical results in Section \ref{illustration} with the help of a non-convex function, defined on a two-dimensional convex region, that arises naturally in applications.

Since our explorations of convexity, or lack of it, rely on the Hessian and its eigenvalues, the function under consideration must have second partial derivatives. Of course, this requirement may not always be fulfilled, like for example in the two limiting cases  $\beta \downarrow -\infty $ and $\beta \uparrow \infty $ of the function $h_{\beta }$ defined by equation (\ref{func-graphing}), which give rise to non-differentiable functions $\min\{g(x),g(y)\}$ and $\max\{g(x),g(y)\}$, respectively. To accommodate such functions, the herein developed method requires a modification, which we posit as a future problem.

It should also be noted that, from the practical point of view, the aforementioned non-differentiable limiting cases may not be necessary as one could work with functions like $h_{\beta }$ for very small (e.g., $\beta=-1000$) or very large (e.g., $\beta=1000$) values of $\beta $, depending on the problem at hand. Though this weakening of the problem alleviates the issue associated with the existence of second partial derivatives, it comes with the necessity of employing considerable computing power.

\section*{Acknowledgments}

The research has been supported by a grant from the Natural Sciences and Engineering Research Council (NSERC) of Canada.

\end{document}